\documentclass[11pt]{amsart}

\usepackage[margin=1.15in]{geometry}
\usepackage{amsmath,amssymb,amsthm,mathtools}
\usepackage{microtype}
\usepackage{enumitem}
\usepackage[hidelinks]{hyperref}
\numberwithin{equation}{section}

\newtheorem{theorem}{Theorem}[section]
\newtheorem{proposition}[theorem]{Proposition}
\newtheorem{lemma}[theorem]{Lemma}
\newtheorem{corollary}[theorem]{Corollary}
\newtheorem{remark}[theorem]{Remark}

\newcommand{\Z}{\mathbb Z}
\newcommand{\R}{\mathbb R}
\newcommand{\T}{\mathbb T}
\newcommand{\one}{\mathbf 1}
\newcommand{\eps}{\varepsilon}
\newcommand{\cA}{\mathcal A}
\newcommand{\ip}[2]{\langle #1,#2\rangle}
\newcommand{\dist}{\operatorname{dist}}
\newcommand{\Aff}{\operatorname{Aff}}
\newcommand{\Span}{\operatorname{span}}
\newcommand{\e}{\mathrm e}

\title[Sharp $\ell^p$-Improving Estimates]
{Sharp $\ell^p$-Improving Estimates for Fixed-Radius
Discrete Spherical Averages}

\author{Rui Han}
\address{Department of Mathematics, Louisiana State University,
Baton Rouge, Louisiana 70803-4918, USA}
\email{rhan@lsu.edu}

\author{Fan Yang}
\address{Department of Mathematics, Louisiana State University,
Baton Rouge, Louisiana 70803-4918, USA}
\email{yangf@lsu.edu}

\date{}

\begin{document}

\begin{abstract}

Let $d\geq 4$ and let $R>0$. When $d=4$, assume that $R^2\in\mathbb{N}\setminus 4\mathbb{N}$; when $d\geq 5$, let $R^2\in\mathbb{N}$ be arbitrary. We prove the fixed-radius estimate 
$$\|A_R f\|_{\ell^{p'}(\mathbb{Z}^d)}\leq C_{d,p,\varepsilon}R^{-d(2/p-1)+\varepsilon}\|f\|_{\ell^p(\mathbb{Z}^d)}$$ for $(d+2)/d\leq p\leq 2$, where $p'$ is the Hölder conjugate exponent of $p$ and $A_R$ is the probability average over the lattice sphere of radius $R$. This extends the fixed-radius estimates of Kesler--Lacey and Hughes to the sharp lower endpoint $p=(d+2)/d$.

\end{abstract}

\maketitle

\section{Introduction}

For $n=(n_1,\ldots,n_d)\in\Z^d$, let
\[
 |n|:=\bigl(n_1^2+\cdots+n_d^2\bigr)^{1/2}.
\]
For $R>0$ with $R^2\in\mathbb N$, define the discrete sphere with radius $R$:
\[
 S_R:=\{n\in\Z^d:|n|=R\}.
\]

In dimensions $d\geq5$, one has the uniform estimate
\begin{equation}\label{eq:sphere-cardinality}
 |S_R|\asymp_d R^{d-2}.
\end{equation}
This estimate goes back to Hardy's circle-method formula for representations
as sums of five or more squares \cite{Hardy}; see also
\cite{MSW,HughesSparse} for modern formulations in the present normalization.

In dimension four, Jacobi's four-square formula gives
\begin{equation}\label{eq:four-square-cardinality}
 |S_R|=8\sum_{\substack{a\mid R^2\\4\nmid a}}a.
\end{equation}
Consequently, if $R^2\in\mathbb N\setminus4\mathbb N$, then
\begin{equation}\label{eq:four-square-lower}
 |S_R|\geq8R^2.
\end{equation}
There is no uniform upper bound for $|S_R|/R^2$ over these levels.  Indeed,
if $R^2$ is squarefree and odd, then
\[
 \frac{|S_R|}{R^2}
 =8\prod_{p\mid R^2}\left(1+\frac1p\right),
\]
which is unbounded along products of the first odd primes.  For
\eqref{eq:four-square-cardinality}, see
\cite[Theorem~1]{Hirschhorn}; see also \cite{Grosswald}.

Let $A_R$ be the probability average
\[
 A_Rf(x):=\frac{1}{|S_R|}\sum_{n\in S_R}f(x-n),
\]
and introduce the circle-method normalization
\begin{equation}\label{eq:calA-def}
 \cA_Rf(x):=R^{2-d}\sum_{n\in S_R}f(x-n).
\end{equation}

For $d\geq5$, \eqref{eq:sphere-cardinality} makes estimates for $A_R$ and
$\cA_R$ equivalent up to constants depending only on $d$.  In dimension
four,
\[
 A_R=\frac{R^2}{|S_R|}\cA_R\leq \frac{1}{8} \cA_R,
\]
whenever $R^2\notin4\mathbb N$, but the reverse comparison is not uniform.
We therefore use $\cA_R$ in the circle-method and incidence arguments,
transfer the restricted endpoint estimate to $A_R$ by this one-sided
comparison, and perform the final interpolation directly for $A_R$.

\begin{theorem}\label{thm:main}
Let $d\geq4$ and let $R>0$.  If $d=4$, assume that
$R^2\in\mathbb N\setminus4\mathbb N$; if $d\geq5$, assume only that
$R^2\in\mathbb N$.  Let
\[
 \frac{d+2}{d}\leq p\leq2.
\]
Then for every $\eps>0$,
\begin{equation}\label{eq:main-estimate}
 \|A_Rf\|_{\ell^{p'}(\Z^d)}
 \lesssim_{d,p,\eps}
 R^{-d(\frac2p-1)+\eps}\|f\|_{\ell^p(\Z^d)}.
\end{equation}
The implicit constant is uniform over all $R^2$ satisfying the
stated condition.  The exponent and the lower endpoint are sharp, see
Remark~\ref{rem:sharpness} below.
\end{theorem}

Throughout the proof of Theorem~\ref{thm:main}, the dimension and radius
satisfy the hypotheses above.

\begin{remark}[Sharpness of the exponent and range]
\label{rem:sharpness}
Let $L=\lceil R\rceil$ and set
\[
 B=[-3L,3L]^d\cap\Z^d,
 \qquad B_0=[-L,L]^d\cap\Z^d.
\]
Then $A_R\one_B=1$ on $B_0$, and therefore
\[
 \|A_R\|_{\ell^p\to\ell^{p'}}
 \gtrsim_d
 \frac{|B_0|^{1/p'}}{|B|^{1/p}}
 \asymp_d R^{-d(2/p-1)}.
\]
Thus the exponent $d(2/p-1)$ in \eqref{eq:main-estimate} cannot be
replaced by any larger exponent uniformly in $R$, at any $p$ in the
stated range.

The following point-mass example forces the lower endpoint for $p$.  For
$d\geq5$, \eqref{eq:sphere-cardinality} gives $|S_R|\asymp R^{d-2}$.
In dimension four, the lower bound
\eqref{eq:four-square-lower} alone does not suffice for this example, so
we take $R^2$ to run through the odd primes.
Then \eqref{eq:four-square-cardinality} gives
$|S_R|=8(R^2+1)\asymp R^2$.
Thus, in either case, along an admissible
sequence of radii,
\[
 A_R\delta_0=|S_R|^{-1}\one_{S_R},
 \qquad
 \|A_R\delta_0\|_{\ell^{p'}}
 =|S_R|^{-1/p}
 \asymp R^{-(d-2)/p}.
\]
If
\[
 p<\frac{d+2}{d}
 \qquad\text{and}\qquad
 0<\eps<\frac{d+2}{p}-d,
\]
then
\[
 R^{d(2/p-1)-\eps}
 \|A_R\delta_0\|_{\ell^{p'}}
 \asymp
 R^{(d+2)/p-d-\eps}
 \longrightarrow\infty.
\]
Hence estimates with the decay in \eqref{eq:main-estimate}, for every
$\eps>0$, cannot extend below $p=(d+2)/d$.  At this endpoint, the large-box
and point-mass examples force exactly the same power of $R$.

In dimension four, some restriction on the radii is necessary for the
normalized averages $A_R$.  If $R=2^j$ with $j\geq1$, then
\eqref{eq:four-square-cardinality} gives $|S_R|=24$, so
$\|A_R\delta_0\|_{\ell^3}=24^{-2/3}$ has no decay as $j\to\infty$.
This is the arithmetic obstruction reflected in the restriction to
$\Lambda_4$ in Kesler--Lacey \cite{KeslerLacey}.

\end{remark}

\begin{corollary}[Fixed-distance incidences]\label{cor:fixed-distance-incidence}

Let $d\geq4$ and let $R>0$.  If $d=4$, assume that
$R^2\in\mathbb N\setminus4\mathbb N$; if $d\geq5$, assume only that
$R^2\in\mathbb N$.  Let $E,F\subset\mathbb Z^d$ be finite, and define
\[
 I_R(E,F):=\#\{(e,f)\in E\times F:|e-f|=R\}.
\]
Then for every $\eps>0$,
\begin{equation}\label{eq:fixed-distance-incidence}
 I_R(E,F)
 \lesssim_{d,\eps}
 R^{\frac{2(d-2)}{d+2}+\eps}
 (|E||F|)^{\frac d{d+2}}.
\end{equation}
In particular,
\[
 I_R(E,E)
 \lesssim_{d,\eps}
 R^{\frac{2(d-2)}{d+2}+\eps}|E|^{\frac{2d}{d+2}}.
\]

\end{corollary}

This is the number of edges in the bipartite distance graph determined by
$E$ and $F$, where two vertices are joined precisely when they are separated
by the prescribed distance $R$.
The estimate has the correct scale, up to the $R^\eps$ loss.  For $d\geq5$,
taking $E=\{0\}$ and $F=S_R$ gives
$I_R(E,F)=|S_R|\asymp R^{d-2}$, while the right-hand side has the same power
of $R$.  The same conclusion holds in dimension four for odd-prime values of $R^2$.  
The estimate also controls
popular centers.  Indeed, if
\[
 r_E(x):=\#\{e\in E:|x-e|=R\},
 \qquad
 F_t:=\{x\in\Z^d:r_E(x)\geq t\},
\]
then $t|F_t|\leq I_R(E,F_t)$, and \eqref{eq:fixed-distance-incidence} implies
\[
 |F_t|
 \lesssim_{d,\eps}
 R^{d-2+\eps}|E|^{d/2}t^{-(d+2)/2}.
\]
Thus there cannot be too many centers whose radius-$R$ spheres contain many
points of $E$.

The surrounding discrete spherical maximal theory begins with Magyar's work
and the theorem of Magyar--Stein--Wainger, with the restricted weak endpoint
due to Ionescu \cite{MagyarMaximal,MSW,IonescuEndpoint}.  Subsequent work
treated sparse families of radii, maximal improving and sparse estimates,
lacunary radii, and sparse domination for the full maximal operator
\cite{HughesSparse,KeslerMaximal,KLMLacunary,KLM}.

For the fixed-radius averages considered here, Kesler--Lacey proved the
estimate
\[
 \|A_Rf\|_{\ell^{p'}}
 \lesssim_{d,p,\omega(R^2)}
 R^{-d(2/p-1)}\|f\|_{\ell^p},
 \qquad
 \frac{d+1}{d-1}<p\leq2,
\]
with no $R^\eps$ factor, but with a constant depending on
$\omega(R^2)$, the number of distinct prime factors of $R^2$
\cite[Theorem~1.1]{KeslerLacey}.  Hughes proved the
corresponding estimate for
$(d+1)/(d-1)\leq p\leq2$, with an $R^\eps$ loss and a constant uniform in
$R^2$ \cite[Theorem~1.1]{Hughes}.

In dimension four, Kesler--Lacey allow
$R^2\in\mathbb N\setminus4\mathbb N$
\cite[Theorem~1.1]{KeslerLacey}, whereas Hughes restricts to odd $R^2$
\cite[Theorem~1.1]{Hughes}.  Theorem~\ref{thm:main} reaches the
sharp endpoint $p=3/2$ throughout the Kesler--Lacey class
$R^2\in\mathbb N\setminus4\mathbb N$.

Discrete $\ell^p$-improving estimates have also been developed for several
other arithmetic averaging families.  Han--Lacey--Yang proved a local
scale-free estimate for averages along the square integers in the sharp
range $3/2<p\leq2$, up to the endpoint, together with sparse bounds for the
associated maximal operator \cite{HanLaceyYangSquares}.
Han--Krause--Lacey--Yang obtained scale-free improving estimates and sparse
bounds for averages along the primes
\cite{HanKrauseLaceyYangPrimes}; endpoint estimates, including logarithmic
density bounds and corresponding sparse bounds, were subsequently obtained
by Lacey--Mousavi--Rahimi \cite{LaceyMousaviRahimi}.  Giannitsi proved
scale-free improving estimates for divisor-function-weighted averages for
$1<p<2$, together with sparse bounds for the associated maximal operator
\cite{GiannitsiDivisor}.

Han--Kova\v{c}--Lacey--Madrid--Yang proved improving estimates for
integer-valued polynomial averages, using Vinogradov mean-value estimates
for the higher-degree moment curve \cite{HanKovacLaceyMadridYang}.  Dasu--Demeter--Langowski obtained the sharp
$\ell^p$-improving range for the discrete paraboloid in every dimension at
least two \cite{DasuDemeterLangowski}.  Dendrinos--Hughes--Vitturi obtained
subcritical restricted weak-type estimates for discrete polynomial curves
\cite{DendrinosHughesVitturi}.

Related arithmetic techniques also appear in the theory of discrete
fractional integrals and singular Radon transforms
\cite{SteinWaingerFractional,IonescuWainger,PierceWeyl}, while related
sparse methods for discrete singular operators appear in
\cite{CuliucKeslerLacey}.

\noindent\textbf{Restricted weak-type formulation.}
At the endpoint $p=(d+2)/d$, the desired global restricted estimate is
\begin{equation}\label{eq:global-restricted}
 \langle\cA_R\one_E,\one_F\rangle
 \lesssim_{d,\eps}
 R^{-\frac{d(d-2)}{d+2}+\eps}
 (|E||F|)^{\frac d{d+2}}
\end{equation}
for finite sets $E,F\subset\Z^d$.  This estimate may be localized to cubes
of sidelength comparable to $R$.  Indeed, partition $\Z^d$ into half-open
cubes $Q_j$ of sidelength $\lceil R\rceil$, let $Q_j^*$ be fixed
enlargements that contain every point at distance $R$ from $Q_j$, and set
\[
 E_j:=E\cap Q_j^*,\qquad F_j:=F\cap Q_j.
\]
The sets $F_j$ partition $F$, the cubes $Q_j^*$ have bounded overlap, and
the support of the spherical kernel gives
\[
 \langle\cA_R\one_E,\one_F\rangle
 =
 \sum_j\langle\cA_R\one_{E_j},\one_{F_j}\rangle.
\]
Moreover, H\"older's inequality and bounded overlap imply
\[
 \sum_j
 |E_j|^{\frac d{d+2}}
 |F_j|^{\frac d{d+2}}
 \lesssim_d
 (|E||F|)^{\frac d{d+2}}.
\]
It therefore suffices to fix one such pair $Q\subset Q^*$, take
\[
 E\subset Q^*,\qquad F\subset Q,
\]
and prove the corresponding local estimate.

This localization is used essentially in the circle-method part of the
proof.  If $x\in F\subset Q$ and $y\in E\subset Q^*$, then
$|x-y|\lesssim_dR$, and hence
\[
 \bigl||x-y|^2-R^2\bigr|\lesssim_d R^2.
\]
This is precisely what allows the divisor factor arising from the nonzero
Ramanujan sums to be absorbed into an $R^\eps$ loss; see Lemma~\ref{lem:ramanujan-sum}.  The localization also
makes $|E|/R^d$ and $|F|/R^d$ genuine bounded local densities.

Set the product density of $E$ and $F$ to be
\begin{equation}\label{eq:local-density-def}
 \delta:=\frac{|E||F|}{R^{2d}}.
\end{equation}

\begin{proposition}[Local restricted endpoint]\label{prop:local-endpoint}
Under this local setup, for every $\eps>0$,
\begin{equation}\label{eq:local-target-intro}
 R^{-d}\langle\cA_R\one_E,\one_F\rangle
 \lesssim_{d,\eps}
 R^\eps\delta^{\frac d{d+2}}.
\end{equation}
\end{proposition}

\noindent\textbf{Strategy and new ingredients.}
The proof combines a circle-method estimate above an explicit density
threshold with rank-sensitive geometric counting estimates in the
complementary low-density range.

The circle-method part of the argument begins with a fixed-radius refinement
of the low-frequency estimate.  Away from the exact lattice sphere, we use
cancellation in the Ramanujan sums; on the exact sphere, the remaining
contribution is a small multiple of the original positive spherical average
and can be absorbed.  The underlying separation and exact-sphere absorption
mechanism appeared in the variable-radius sparse-maximal argument of
Kesler--Lacey--Mena \cite{KLM}.  Here we develop a cumulative
fixed-radius version, summing all moduli $q\leq N$ together.
Combined with the high-frequency $\ell^2$ estimate, this proves
\eqref{eq:local-target-intro} whenever the local product density $\delta$ is not too small.

The complementary low-density range is handled by geometric incidence counting. The broader strategy of reducing restricted weak-type estimates for positive averaging operators to geometric or combinatorial incidence estimates is classical; see Christ's work \cite{ChristRefinements}.  For curved hypersurface averages, including the spherical case, Schlag cast
restricted weak-type estimates into a continuum-incidence framework and
controlled the measure of the set of admissible centers through
simplex-volume estimates \cite{SchlagContinuum}.

The low-density argument in the present paper shares this broad
incidence-counting philosophy, but the configurations being counted and the
mechanism used to control them are different.  We study integer moments of
the fixed-distance incidence function $r_E$.  Expanding such a moment
produces configurations of lattice points lying on a common radius-$R$
sphere.  These configurations are organized by affine rank: after a
basis $B$ is fixed, the remaining points lie on its circumsphere
$\Sigma_B$, while the admissible centers lie on the common-center sphere
$X_B$.

The same rank-sensitive integer-moment estimate underlies both parity cases,
but it enters the incidence estimate differently.  When $d$ is even, the
natural endpoint exponent $(d+2)/2$ is an integer, and H\"older's inequality
produces exactly the required power $|F|^{d/(d+2)}$.  In dimensions
$d\geq6$, the resulting moment estimate covers the full low-density range
after interchanging $E$ and $F$. When $d=4$, the third moment gives
the desired bound when the smaller set has size $\lesssim R^2$, but failure
of the circle-method density condition only yields $|E||F|\lesssim R^5$.
We handle the remaining window using Mudgal's point--sphere incidence estimate
\cite{Mudgal}.

When $d$ is odd, the natural exponent $(d+2)/2$ is half-integral.  For odd
$d\geq7$, it can be treated directly, at the cost of an additional
square-root estimate. In dimension five, however, the half-integral approach does not close and would require a separate argument. To avoid such a dimension-specific treatment, we instead use the next integer moment, $(d+3)/2$.
H\"older's inequality at this exponent produces a power of
$|F|$ slightly larger than the endpoint power.  We therefore retain the
factor involving $|F|$ and estimate it jointly with the affine-rank
contributions to the integer moment, using the balance between $E$ and
$F$ and the low-density product bound.  This yields the desired estimate
for every odd dimension $d\geq5$, with dimension five as the critical
top-rank case.

\par\medskip
\noindent\textbf{Proof structure.}
The proof of Proposition~\ref{prop:local-endpoint} is organized around four
components.
\begin{enumerate}[
 label=\textup{\Alph*.},
 leftmargin=2.2em,
 itemsep=0.35em
]
\item \emph{The circle-method estimate above the density threshold.}
Theorem~\ref{thm:analytic-density} combines the refined low-frequency
estimate with the high-frequency $\ell^2$ estimate and proves
\eqref{eq:local-target-intro} whenever
\[
 \delta\geq R^{-\frac{(d-3)(d+2)}{d-2}}.
\]

\item \emph{The even-dimensional low-density estimate.}
Theorem~\ref{thm:geometric-density} proves the endpoint moment estimate and
the resulting incidence bound in even dimensions $d\geq4$ when the smaller
set has size at most a fixed multiple of $R^{d/2}$.

\item \emph{The four-dimensional bridge estimate.}
Theorem~\ref{thm:d4-bridge} uses Mudgal's point--sphere incidence bound to
cover the remaining four-dimensional low-density window.

\item \emph{The odd-dimensional low-density estimate.}
Theorem~\ref{thm:odd-geometric-density} proves the desired incidence bound
in every odd dimension $d\geq5$.
\end{enumerate}

These four components cover all density regimes with the same threshold
constant.  If the circle-method condition fails, then
\[
 \delta<R^{-\kappa_d},
 \qquad
 |E||F|<R^{2d-\kappa_d}.
\]
When $d=4$, one has $\kappa_4=3$, and hence
\[
 |E||F|<R^5.
\]
After interchanging $E$ and $F$ if necessary,
Theorem~\ref{thm:geometric-density} applies when $|E|\leq R^2$, while
Theorem~\ref{thm:d4-bridge} applies when $|E|>R^2$.

For even $d\geq6$, one has
\[
 \kappa_d-d=\frac{d-6}{d-2}\geq0,
\]
and hence
\[
 |E||F|<R^d,
 \qquad
 \min\{|E|,|F|\}<R^{d/2}.
\]
Thus Theorem~\ref{thm:geometric-density} applies after interchanging $E$ and
$F$ if necessary.  In odd dimensions,
Theorem~\ref{thm:odd-geometric-density} applies to the product bound above.

\noindent\textbf{Organization of the paper.}

Section~2 provides the preliminary notation and the lattice-point estimate
for embedded spheres.  Section~3 introduces the circle-method decomposition,
establishes the high-frequency estimate, and derives the exact
Ramanujan-sum description of the low-frequency kernel.  Section~4 proves
the refined low-frequency estimate.  Section~5 combines these analytic
estimates to prove the endpoint bound above the circle-method density
threshold.

Section~6 develops the common-sphere geometry and proves the
rank-sensitive integer-moment estimate used in both parity cases.
Section~7 proves the even-dimensional low-density estimate.  Section~8
proves the four-dimensional bridge estimate, and Section~9 proves the
odd-dimensional low-density estimate.

Section~10 combines the analytic and geometric estimates to prove the local
restricted endpoint and then deduces the corresponding global restricted
estimate.  Section~11 passes from the restricted endpoint to strong type,
interpolates throughout the range ${d+2}/{d}\leq p\leq2$, and completes the
proofs of Theorem~\ref{thm:main} and
Corollary~\ref{cor:fixed-distance-incidence}.

\section{Preliminaries}

We write
\[
 \e(t):=e^{2\pi i t},\qquad \e_q(t):=\e(t/q).
\]
For $f\in\ell^1(\Z^d)$, our Fourier transform is
\[
 \widehat f(\xi)=\sum_{x\in\Z^d}f(x)\e(-x\cdot\xi),
 \qquad \xi\in\T^d,
\]
and the inverse Fourier coefficient of a multiplier $m$ is
\[
 \check m(n)=\int_{\T^d}m(\xi)\e(n\cdot\xi)\,d\xi.
\]
For a function or finite measure on $\R^d$, the same notation is used with the
Euclidean Fourier transform.  We write
\[
 \tau(m):=\sum_{r\mid m}1
\]
for the divisor function, $\mu$ for the M\"obius function, and $\phi$ for
Euler's totient function.

We require the following uniform lattice-point estimate.

\begin{lemma}[Lattice points on embedded spheres]\label{lem:embedded-sphere}
Let $1\leq k\leq d-1$, and let $\Sigma$ be a nondegenerate
$k$-dimensional Euclidean sphere embedded in $\R^d$, with arbitrary center
and affine span, and with actual radius $\rho>0$.  Then for every $\eps>0$,
\begin{equation}\label{eq:embedded-sphere-actual}
 \#(\Sigma\cap\Z^d)
 \lesssim_{d,\eps}1+\rho^{k-1+\eps}.
\end{equation}
In particular, if $\rho\leq R$ and $R\geq1$, then
\begin{equation}\label{eq:embedded-sphere-bound}
 \#(\Sigma\cap\Z^d)\lesssim_{d,\eps}R^{k-1+\eps}.
\end{equation}
For a radius-zero section or a zero-dimensional sphere, the corresponding
count is $O(1)$.
\end{lemma}

For $\rho>1$, \eqref{eq:embedded-sphere-actual} is Lemma~4 of
Huang--Zhang \cite{HuangZhang}, after relabeling the dimension
of the embedded sphere.  The statement is uniform in the sphere, so its
constant is independent of the center, affine span, orientation, and
embedding; in particular, no determinant, covolume, or rationality factor is
present.  For $0\leq\rho\leq1$, lattice separation gives $O_d(1)$ points.
Notice in particular that an embedded circle of radius at most $R$ contains
$O_{d,\eps}(R^\eps)$ lattice points.

\section{The circle-method decomposition}

The analytic estimates in this section and Sections~4--5 are used for every
$d\geq4$.  In dimension four, they are applied only when
$R^2\in\mathbb N\setminus4\mathbb N$.

\noindent\textbf{The low-frequency mechanism.}
In the variable-radius sparse-maximal setting,  Kesler--Lacey--Mena group moduli in a dyadic block, keep the resulting
Ramanujan-sum factor explicit, separate its zero argument from the
nonzero arguments, and absorb the exact-sphere term; see
\cite{KLM}.  We implement the same
mechanism at one fixed radius while summing all moduli $q\leq N$ together in this and the next two sections, which leads to the
precise $N$-dependence required in Theorem~\ref{thm:analytic-density}.

We follow Magyar--Stein--Wainger \cite{MSW}, absorbing the dimensional
constant into the surface measure.  Let
\[
 c_d:=\frac{\pi^{d/2}}{\Gamma(d/2)},
\]
and let $d\sigma_R$ be $c_d$ times probability surface measure on
$\{x\in\R^d:|x|=R\}$. 
Thus
\[
 d\sigma_R
 =R^{2-d}\delta(|x|^2-R^2)\,dx,\qquad
 d\sigma_R(\R^d)=c_d,\qquad
 \widehat{d\sigma_R}(\xi)=\widehat{d\sigma_1}(R\xi).
\]

Let $\Psi_H$ be the fixed cutoff in the approximation formula
\cite{HughesSparse}.  Hughes writes the lattice Fourier series with
character $\e(x\cdot\xi)$, whereas our multiplier convention uses
$\e(-x\cdot\xi)$, so we set
\[
 \Phi(\theta):=\Psi_H(-\theta),\qquad
 \Phi_t(\theta):=\Phi(t\theta),\qquad t\geq1.
\]
This fixed cutoff belongs to $C_c^\infty((-1/4,1/4)^d)$ and equals one
on $(-1/8,1/8)^d$.  All implicit constants below may depend on $\Phi$.
For $q\in\mathbb N$, write $\Z_q:=\Z/q\Z$ and let $\Z_q^\times$ denote its
group of units; for $q=1$ we use the convention $\Z_1^\times=\{0\}$.  For
$\ell\in\Z^d$, define
\begin{equation}\label{eq:K-def}
 K(R,q,\ell)
 :=q^{-d}\sum_{a\in\Z_q^\times}\e_q(-aR^2)
 \sum_{u\in\Z_q^d}\e_q\bigl(a|u|^2+u\cdot\ell\bigr).
\end{equation}
The quantity $K(R,q,\ell)$ is $q\Z^d$-periodic in $\ell$.  For $q\leq R$,
define the full main-term multiplier and its convolution operator by
\[
 c_{R,q}(\xi):=\sum_{\ell\in\Z^d}K(R,q,\ell)
 \Phi_q(\xi-\ell/q)\widehat{d\sigma_R}(\xi-\ell/q),
 \qquad C_{R,q}f:=\mathcal F^{-1}(c_{R,q}\widehat f).
\]
The sum over $\ell\in\Z^d$ makes $c_{R,q}$ a $\Z^d$-periodic multiplier
on $\T^d$.

With the reflected cutoff above, the decomposition in
\cite[(10)--(13)]{HughesSparse}, evaluated at $-\xi$, gives the exact identity
\[
\cA_R=\sum_{1\leq q\leq R}C_{R,q}+\mathcal E_R,
\]
and \cite[(11)]{HughesSparse} gives
\begin{equation}\label{eq:ER-L2}
\|\mathcal E_R\|_{\ell^2\to\ell^2}
\lesssim_{d,\eps}R^{(3-d)/2+\eps}.
\end{equation}

Fix an integer $N$ with $1\leq N\leq R$.  For $1\leq q\leq N$, let
\begin{equation}\label{eq:c1-def}
 c^1_{R,q}(\xi)
 :=\sum_{\ell\in\Z^d}K(R,q,\ell)
 \Phi_{Rq/N}(\xi-\ell/q)
 \widehat{d\sigma_R}(\xi-\ell/q),
\end{equation}
where the locally finite sum is a $\Z^d$-periodic multiplier.  Since
$Rq/N\geq q$, the cutoff $\Phi_{Rq/N}$ is supported in a smaller
neighborhood of each rational point $\ell/q$ than $\Phi_q$.  Indeed, if $\theta\in\operatorname{supp}\Phi_{Rq/N}$, then
\[
 \|q\theta\|_\infty<\frac{N}{4R}.
\]
Thus, whenever $N\leq R/4$, this support lies in the region
$\|q\theta\|_\infty<1/8$ on which $\Phi_q=1$.
Thus $C^1_{R,q}$ is the narrow low-frequency portion of $C_{R,q}$.
Let
$C^1_{R,q}$ be the associated convolution operator and put
\begin{equation}\label{eq:LN-def}
 L_N:=\sum_{q\leq N}C^1_{R,q}.
\end{equation}
Finally, define
\[
 C^2_{R,q}:=C_{R,q}-C^1_{R,q},\qquad
 H_N:=\mathcal E_R+\sum_{N<q\leq R}C_{R,q}+\sum_{q\leq N}C^2_{R,q}.
\]
By definition,
\begin{equation}\label{eq:decomposition}
 \cA_R=L_N+H_N.
\end{equation}

\subsection{The high-frequency estimate}

Write $q=q_{\rm odd}q_{\rm even}$, where $q_{\rm even}$ is the full power of
two dividing $q$, and set
\begin{equation}\label{eq:rho-def}
 \rho_R(q):=\bigl((q_{\rm odd},R^2)q_{\rm even}\bigr)^{1/2}.
\end{equation}
Together with the residual estimate \eqref{eq:ER-L2}, we use the
explicit odd/even Weil estimate
\begin{equation}\label{eq:Weil-correct}
 \sup_{\ell}|K(R,q,\ell)|
 \lesssim_{d,\eps}q^{-(d-1)/2+\eps}\rho_R(q),
\end{equation}
as in \cite[(18)]{HughesSparse}.
We also use the standard stationary-phase bound
\begin{equation}\label{eq:sigma-stationary}
 |\widehat{d\sigma_R}(\xi)|
 \lesssim_d(1+R|\xi|)^{-(d-1)/2}.
\end{equation}

The following estimates follow from Proposition~2.6 of Kesler--Lacey \cite{KeslerLacey}, after rewriting the
notation.

\begin{lemma}[Divisor sums for the Weil factor]\label{lem:rho-sums}
Let $1\leq N<R$, $a>1$, and $\eps>0$.  Then
\begin{align}
 \sum_{q>N}q^{-a}\rho_R(q)
 &\lesssim_a N^{1-a}\sigma_{-1/2}(R^2),
 \label{eq:rho-large-sum}\\
 \sum_{q\leq N}q^\eps\rho_R(q)
 &\lesssim_\eps N^{1+\eps}\sigma_{-1/2}(R^2),
 \label{eq:rho-small-sum}
\end{align}
where
\[
 \sigma_{-1/2}(m):=\sum_{r\mid m}r^{-1/2}.
\]
\end{lemma}

The proof below follows the decomposition of $M_{2,2}$ and $M_{2,3}$ in
\cite[Section~3]{KeslerLacey}.  We include the details, since the precise
powers of $q$ are needed in dimension four.

\begin{proposition}[High-frequency $\ell^2$ estimate]\label{prop:HN-L2}
There is $c=c(d,\Phi)>0$, with $c\leq1/4$, such that, for
$1\leq N\leq cR$ and every $\eps>0$,
\begin{equation}\label{eq:HN-L2}
 \|H_N\|_{\ell^2\to\ell^2}
 \lesssim_{d,\eps}R^\eps N^{-(d-3)/2}.
\end{equation}
\end{proposition}

\begin{proof}
It suffices to consider $0<\eps\leq1$.  For $q>N$, the disjointness of the
fixed-$q$ rational supports, Plancherel's theorem, and
\eqref{eq:Weil-correct} give
\[
 \|C_{R,q}\|_{2\to2}
 \lesssim_{d,\eps}
 q^{-(d-1)/2+\eps/4}\rho_R(q).
\]
Lemma~\ref{lem:rho-sums}, applied with
$(d-1)/2-\eps/4>1$, therefore yields
\[
 \left\|\sum_{N<q\leq R}C_{R,q}\right\|_{2\to2}
 \lesssim_{d,\eps}
 N^{-(d-3)/2+\eps/4}\sigma_{-1/2}(R^2).
\]

For $q\leq N$, the multiplier of $C^2_{R,q}$ is supported where
$R|\xi-\ell/q|_\infty\gtrsim N/q$.  Hence
\eqref{eq:Weil-correct}, \eqref{eq:sigma-stationary}, and Plancherel's
theorem give
\[
 \|C^2_{R,q}\|_{2\to2}
 \lesssim_{d,\eps}
 N^{-(d-1)/2}q^{\eps/4}\rho_R(q).
\]
A second application of Lemma~\ref{lem:rho-sums} gives
\[
 \left\|\sum_{q\leq N}C^2_{R,q}\right\|_{2\to2}
 \lesssim_{d,\eps}
 N^{-(d-3)/2+\eps/4}\sigma_{-1/2}(R^2).
\]
Combining these estimates with \eqref{eq:ER-L2}, and using
\[
 \sigma_{-1/2}(R^2)\lesssim_\eps R^{\eps/4},
 \qquad N^{\eps/4}\leq R^{\eps/4},
\]
proves \eqref{eq:HN-L2}.
\end{proof}

\subsection{The exact Ramanujan kernel}

The representation below is classical.  Hughes writes the resulting
Ramanujan kernel explicitly in \cite[(39)]{HughesSparse}.
Kesler--Lacey \cite[Proposition~2.11]{KeslerLacey} compute, for each fixed
reduced numerator $a\in\Z_q^\times$, the physical-space phase
$\e_q(a(|n|^2-R^2))$; summing those phases over $a$ is exactly the Ramanujan
sum $c_q(|n|^2-R^2)$.  We include the calculation because the surface
normalization and torus periodization are essential for an exact identity.

For $m\in\Z$, define the Ramanujan sum
\[
 c_q(m):=\sum_{a\in\Z_q^\times}\e_q(am).
\]

\begin{lemma}[Exact low-frequency kernel]\label{lem:ramanujan-kernel}
For every $n\in\Z^d$,
\begin{equation}\label{eq:exact-kernel}
 \check C^1_{R,q}(n)
 =c_q(|n|^2-R^2)
 \bigl(\check\Phi_{Rq/N}*d\sigma_R\bigr)(n).
\end{equation}
\end{lemma}

\begin{proof}
Write $\ell=\ell_0+qv$ with $\ell_0\in\Z_q^d$ and $v\in\Z^d$.  As explained
above, the sum in \eqref{eq:c1-def} periodizes the Euclidean multiplier
$\Phi_{Rq/N}\widehat{d\sigma_R}$.  Tiling $\R^d$ by fundamental domains and
changing variables therefore gives
\[
 \check C^1_{R,q}(n)
 =\bigl(\check\Phi_{Rq/N}*d\sigma_R\bigr)(n)
 \sum_{\ell\in\Z_q^d}K(R,q,\ell)\e_q(n\cdot\ell).
\]
Using \eqref{eq:K-def} and finite Fourier inversion,
\begin{align*}
 \sum_{\ell\in\Z_q^d}K(R,q,\ell)\e_q(n\cdot\ell)
 &=q^{-d}\sum_{a\in\Z_q^\times}\e_q(-aR^2)
   \sum_{u\in\Z_q^d}\e_q(a|u|^2)
   \sum_{\ell\in\Z_q^d}\e_q((u+n)\cdot\ell)\\
 &=\sum_{a\in\Z_q^\times}\e_q\bigl(a(|n|^2-R^2)\bigr)\\
 &=c_q(|n|^2-R^2).
\end{align*}
Indeed, the innermost sum equals $q^d$ precisely when
$u\equiv-n\pmod q$ and is zero otherwise.  This proves
\eqref{eq:exact-kernel}.
\end{proof}

\subsection{The Euclidean cap estimate}

\begin{lemma}\label{lem:kernel-localization}
For every $A>0$ and every $n\in\Z^d$,
\begin{equation}\label{eq:kernel-localization}
 \left|\bigl(\check\Phi_{Rq/N}*d\sigma_R\bigr)(n)\right|
 \lesssim_{A,d}
 \frac{N}{qR^d}
 \left(1+\frac{N\,\bigl||n|-R\bigr|}{Rq}\right)^{-A}.
\end{equation}
The estimate is uniform for $1\leq q\leq N\leq cR$.
\end{lemma}

\begin{proof}
The scaling identity
\[
 \check\Phi_{Rq/N}(x)
 =\left(\frac{Rq}{N}\right)^{-d}
 \check\Phi\left(\frac{Nx}{Rq}\right)
\]
and the rapid decay of $\check\Phi$ give, for any $B>0$,
\begin{align*}
 \left|\bigl(\check\Phi_{Rq/N}*d\sigma_R\bigr)(n)\right|
 &\lesssim_B
 \left(\frac{Rq}{N}\right)^{-d}
 \int_{S_R}
 \left(1+\frac{N|n-y|}{Rq}\right)^{-B}d\sigma_R(y).
\end{align*}
Put $t:=Rq/N$ and $D:=\dist(n,S_R)$.  Uniformly in
$z\in\R^d$ and $u>0$,
\[
 d\sigma_R(S_R\cap B(z,u))
 \lesssim_d\min\{1,(u/R)^{d-1}\}.
\]
Choose $B>A+d-1$, extract the factor $(1+D/t)^{-A}$, and decompose the
remaining integral into annuli of radii $2^jt$.  For $2^jt\leq R$ use the
cap estimate above.  For $2^jt>R$ use the fixed total mass of $d\sigma_R$;
the resulting tail is bounded by
\[
 (t/R)^{B-A}\leq(t/R)^{d-1}.
\]
Consequently,
\[
 \left|\bigl(\check\Phi_t*d\sigma_R\bigr)(n)\right|
 \lesssim_{A,d}
 t^{-1}R^{-(d-1)}(1+D/t)^{-A}.
\]
Since $\dist(n,S_R)=\bigl||n|-R\bigr|$ and
$[(Rq/N)R^{d-1}]^{-1}=N/(qR^d)$, this is
\eqref{eq:kernel-localization}.
\end{proof}

\subsection{Summing the Ramanujan factors}

\begin{lemma}[Ramanujan summation]\label{lem:ramanujan-sum}
Let $m\in\Z\setminus\{0\}$.  Then
\begin{equation}\label{eq:ramanujan-sum-bound}
 \sum_{q\leq N}\frac{|c_q(m)|}{q}
 \lesssim \tau(|m|)\log(2N).
\end{equation}
Consequently, if $0<|m|\lesssim R^2$ and $N\leq R$, then for every
$\eps>0$,
\begin{equation}\label{eq:ramanujan-epsilon}
 \sum_{q\leq N}\frac{|c_q(m)|}{q}
 \lesssim_\eps R^\eps.
\end{equation}
For $m=0$,
\begin{equation}\label{eq:ramanujan-zero}
 \sum_{q\leq N}\frac{c_q(0)}q
 =\sum_{q\leq N}\frac{\phi(q)}q
 \lesssim N.
\end{equation}
\end{lemma}

\begin{proof}
For $m<0$, all divisors below mean positive divisors of $|m|$.
The exact identity
\[
 c_q(m)=\sum_{r\mid(q,|m|)}r\,\mu(q/r)
\]
gives
\begin{align*}
 \sum_{q\leq N}\frac{|c_q(m)|}{q}
 &\leq\sum_{r\mid |m|}r
    \sum_{\substack{q\leq N\\r\mid q}}\frac1q\\
 &=\sum_{r\mid |m|}\sum_{k\leq N/r}\frac1k
 \lesssim \tau(|m|)\log(2N).
\end{align*}
Since $0<|m|\lesssim_dR^2$ and $N\leq R$, the divisor bound and
$\log(2N)\lesssim_\eps N^{\eps/4}$ give
\[
 \tau(|m|)\log(2N)
 \lesssim_{d,\eps}|m|^{\eps/4}N^{\eps/4}
 \lesssim_{d,\eps} R^\eps.
\]
This proves \eqref{eq:ramanujan-epsilon}.  Finally $c_q(0)=\phi(q)$, which
yields \eqref{eq:ramanujan-zero}.
\end{proof}

Combining Lemmas~\ref{lem:ramanujan-kernel}--\ref{lem:ramanujan-sum}
gives the following zero/nonzero dichotomy.  Put
$m=|n|^2-R^2$.  If $|n|\neq R$ and $|n|\lesssim R$, then
$0<|m|\lesssim R^2$; using the factor $N/(qR^d)$ from
\eqref{eq:kernel-localization} and summing $|c_q(m)|/q$ by
\eqref{eq:ramanujan-epsilon} gives
\begin{equation}\label{eq:off-sphere-kernel}
 \left|\sum_{q\leq N}\check C^1_{R,q}(n)\right|
 \lesssim_{d,\eps}R^\eps\frac{N}{R^d}.
\end{equation}
If $|n|=R$, then $m=0$ and $c_q(0)=\phi(q)$; hence
\eqref{eq:kernel-localization} and \eqref{eq:ramanujan-zero} instead give
\begin{equation}\label{eq:on-sphere-kernel}
 \sum_{q\leq N}|\check C^1_{R,q}(n)|
 \lesssim_d\frac{N^2}{R^d}.
\end{equation}
The larger $N^2/R^d$ contribution occurs only on $S_R$.
In the next proposition, this exact-sphere contribution is controlled by
a small multiple of the original positive spherical average and absorbed.
This is the fixed-radius form of the zero-parameter separation in \cite{KLM}.

\section{The refined low frequency estimate}

We work in the local setup introduced above: let $Q\subset\Z^d$ be a cube
of sidelength comparable to $R$, let $Q^*$ be a fixed enlargement with
$|Q|\asymp_d |Q^*|\asymp_d R^d$, and take
\[
E\subset Q^*,\qquad F\subset Q.
\]
In particular, every difference $x-y$ arising below satisfies
$|x-y|\lesssim_d R$.

\begin{proposition}[Refined low-frequency estimate]\label{prop:low-frequency}
Let $E\subset Q^*$ and $x\in Q$.  After handling the finitely many radii
$R<c^{-1}$ by
trivial bounds, assume $R\geq c^{-1}$.  For $1\leq N\leq cR$, where
$c=c(d,\Phi)>0$ is sufficiently small,
\begin{equation}\label{eq:low-frequency-refined}
 |L_N\one_E(x)|
 \lesssim_{d,\eps}
 R^\eps N\frac{|E|}{R^d}
 +C_d\frac{N^2}{R^2}\cA_R\one_E(x).
\end{equation}
Consequently,
\begin{equation}\label{eq:absorbed-pointwise}
 \cA_R\one_E(x)
 \lesssim_{d,\eps}
 R^\eps N\frac{|E|}{R^d}+|H_N\one_E(x)|.
\end{equation}
\end{proposition}

\begin{proof}
For $x\in Q$ and $y\in Q^*$, the difference $n=x-y$ satisfies $|n|\lesssim_dR$.
Split the convolution kernel of $L_N$ into $|n|\neq R$ and $|n|=R$.
By \eqref{eq:off-sphere-kernel},
\[
 \sum_{\substack{y\in E\\|x-y|\neq R}}
 \left|\sum_{q\leq N}\check C^1_{R,q}(x-y)\right|
 \lesssim_{d,\eps}R^\eps\frac{N}{R^d}|E|
 =R^\eps N\frac{|E|}{R^d}.
\]
By \eqref{eq:on-sphere-kernel},
\begin{align*}
 \sum_{\substack{y\in E\\|x-y|=R}}
 \left|\sum_{q\leq N}\check C^1_{R,q}(x-y)\right|
 &\lesssim_d\frac{N^2}{R^d}
   \sum_{|n|=R}\one_E(x-n)\\
 &=\frac{N^2}{R^2}\cA_R\one_E(x).
\end{align*}
This proves \eqref{eq:low-frequency-refined}.

Using the exact decomposition \eqref{eq:decomposition} and the positivity of
$\cA_R\one_E$,
\begin{align*}
 \cA_R\one_E
 &\leq |L_N\one_E|+|H_N\one_E|\\
 &\leq C_{d,\eps}R^\eps N\frac{|E|}{R^d}
   +C_d\frac{N^2}{R^2}\cA_R\one_E
   +|H_N\one_E|.
\end{align*}
Decrease the previously fixed $c$, if necessary, so that
$C_dc^2\leq1/2$.  This choice depends only on $d$ and $\Phi$, not on
$\eps$.  The last multiple of $\cA_R\one_E$ is then absorbed into the left
side, proving
\eqref{eq:absorbed-pointwise}.
\end{proof}

\begin{proposition}\label{prop:bourgain}
Let $E\subset Q^*$ and $F\subset Q$, and let $\delta$ be as in
\eqref{eq:local-density-def}.
Then, for $1\leq N\leq cR$,
\begin{equation}\label{eq:bourgain-refined}
 R^{-d}\ip{\cA_R\one_E}{\one_F}
 \lesssim_{d,\eps}
 R^\eps\left(
 N\delta+N^{-(d-3)/2}\delta^{1/2}
 \right).
\end{equation}
\end{proposition}

\begin{proof}
Pair \eqref{eq:absorbed-pointwise} with $\one_F$ and divide by $R^d$.  The
low-frequency term is bounded by
\[
 R^\eps N\frac{|E||F|}{R^{2d}}
 =R^\eps N\delta.
\]
For the high-frequency term, Cauchy--Schwarz and \eqref{eq:HN-L2} give
\begin{align*}
 R^{-d}\ip{|H_N\one_E|}{\one_F}
 &\leq R^{-d}\|H_N\one_E\|_2|F|^{1/2}\\
 &\lesssim_{d,\eps}
 R^\eps N^{-(d-3)/2}
 \frac{|E|^{1/2}|F|^{1/2}}{R^d}\\
 &=R^\eps N^{-(d-3)/2}\delta^{1/2},
\end{align*}
as claimed.
\end{proof}

\section{The circle-method estimate above the density threshold}

Define
\begin{equation}\label{eq:kappa-def}
 \kappa_d:=\frac{(d-3)(d+2)}{d-2}.
\end{equation}

\begin{theorem}\label{thm:analytic-density}
Assume $d\geq4$, and in dimension four assume that
$R^2\in\mathbb N\setminus4\mathbb N$.  Under the local setup above, if
\begin{equation}\label{eq:analytic-density-hyp}
 \delta\geq R^{-\kappa_d},
\end{equation}
then the local endpoint estimate \eqref{eq:local-target-intro} holds.
\end{theorem}

\begin{proof}
The localization gives $\delta\leq D_d$ for a fixed constant $D_d$.
After absorbing the finitely many bounded radii, assume that $R$ is large.
Let $c>0$ be the cutoff constant in Proposition~\ref{prop:bourgain}, and fix
$0<c_0<c/4$.

First suppose that
\[
 \delta\geq R^{-(d-1)}.
\]
If $\delta\geq(c_0/2)^{d-1}$, take $N=1$.  Since then
$(c_0/2)^{d-1}\leq\delta\leq D_d$, one has
\[
 \delta+\delta^{1/2}\lesssim_{d,c_0,D_d}\delta^{d/(d+2)},
\]
and \eqref{eq:bourgain-refined} gives the desired estimate.  Otherwise set
\[
 N:=\left\lfloor c_0\delta^{-1/(d-1)}\right\rfloor \asymp_{c_0}\delta^{-1/(d-1)}.
\]
Moreover,
$N\leq c_0R\leq cR$.  Therefore
\begin{align*}
 N\delta+N^{-(d-3)/2}\delta^{1/2}
 \lesssim_{d,c_0}\delta^{(d-2)/(d-1)}
 \leq\delta^{d/(d+2)}.
\end{align*}
The last inequality uses $\delta<1$ and
\[
 \frac{d-2}{d-1}\geq\frac d{d+2},
\]
which holds for every $d\geq4$.

It remains to consider
\[
 R^{-\kappa_d}\leq\delta<R^{-(d-1)}.
\]
When $d=4$, this interval is empty because $\kappa_4=3=d-1$.
For $d\geq5$, take $N=\lfloor c_0R\rfloor$.  For large $R$,
$N\asymp_{c_0}R$ and $1\leq N\leq cR$.  Since
$\delta<R^{-(d-1)}$,
\[
 \frac{R\delta}{\delta^{d/(d+2)}}
 =R\delta^{2/(d+2)}
 \leq R^{(4-d)/(d+2)}\leq1.
\]
Also, because $1/2-d/(d+2)=-(d-2)/(2(d+2))$ and
$\delta\geq R^{-\kappa_d}$,
\begin{align*}
 \frac{R^{-(d-3)/2}\delta^{1/2}}
 {\delta^{d/(d+2)}}
 &=R^{-(d-3)/2}
 \delta^{-(d-2)/(2(d+2))}\\
 &\leq
 R^{-(d-3)/2+
 \kappa_d(d-2)/(2(d+2))}=1.
\end{align*}
Thus both terms in \eqref{eq:bourgain-refined} have the required size.
\end{proof}

\section{Geometric preliminaries}

For a finite set $E\subset\Z^d$, define the fixed-distance incidence function
\begin{equation}\label{eq:rE-def}
 r_E(x):=\sum_{|y|=R}\one_E(x-y)
       =\#\{e\in E:|x-e|=R\}.
\end{equation}
Thus, for every finite $F\subset\Z^d$,
\begin{equation}\label{eq:calA-rE}
 I_R(E,F)=\sum_{x\in F}r_E(x),
 \qquad
 \langle\cA_R\one_E,\one_F\rangle=R^{2-d}I_R(E,F).
\end{equation}

\subsection{Common-sphere geometry}

We use the following elementary description.  Points
$b_0,\ldots,b_r$ are affinely independent when
$b_1-b_0,\ldots,b_r-b_0$ are linearly independent.

\begin{lemma}\label{lem:common-center-geometry}
Let $B=\{b_0,\ldots,b_r\}\subset\R^d$ be affinely independent, put
\[
 V:=\Span\{b_i-b_0:1\leq i\leq r\},\qquad A_B:=b_0+V,
\]
and let $c_B\in A_B$ be the unique point equidistant from the points of $B$.
Write $r_B:=|c_B-b_i|$ and
\[
 \Sigma_B:=\{e\in A_B:|e-c_B|=r_B\}.
\]
The common radius-$R$ centers of $B$ are
\begin{equation}\label{eq:XB-explicit}
 X_B=\{c_B+z:z\in V^\perp,\ |z|^2=R^2-r_B^2\}.
\end{equation}
Equivalently, $x\in X_B$ means that the radius-$R$ sphere centered at $x$
contains every point of $B$ on its boundary.  Thus $X_B$ is an embedded
$(d-r-1)$-sphere if $r_B<R$, a singleton if
$r_B=R$, and empty if $r_B>R$.  Moreover, for $e\in A_B$ and $x\in X_B$,
\[
 |x-e|=R\quad\Longleftrightarrow\quad e\in\Sigma_B.
\]
\end{lemma}

\begin{proof}
Subtracting the squared-distance equations for $b_i$ and $b_0$ shows that a
point is equidistant from all points of $B$ precisely when it belongs to
$c_B+V^\perp$.  Write such a point as $x=c_B+z$, with $z\in V^\perp$.
Since $c_B-b_i\in V$, orthogonality gives
\[
 |x-b_i|^2=|z|^2+|c_B-b_i|^2=|z|^2+r_B^2.
\]
This proves \eqref{eq:XB-explicit} and the description of $X_B$.  The same
calculation, with $e-c_B\in V$, gives
$|x-e|^2=|z|^2+|e-c_B|^2$ and proves the final equivalence.
\end{proof}

For every affine basis $B$ that occurs below, $X_B\neq\varnothing$, and hence
$r_B\leq R$.  If $r=0$, then $\Sigma_B$ is a singleton; if $r=1$, it
has at most two points; and if $r\geq2$, it is an embedded
$(r-1)$-sphere in $A_B$.  Set
\begin{equation}\label{eq:Lr-definition-expanded}
 L_r:=
 \begin{cases}
  1,&r=0,1,\\
  \min\bigl(|E|,R^{r-2}\bigr),&r\geq2.
 \end{cases}
\end{equation}
Then Lemma~\ref{lem:embedded-sphere} gives, for every auxiliary $\eta>0$,
\begin{equation}\label{eq:Lr-bound}
 \#(E\cap\Sigma_B)\lesssim_{d,\eta}R^\eta L_r.
\end{equation}
In the proof of Lemma~\ref{lem:integer-moments} below we use
\eqref{eq:Lr-bound} with $\eta=\eps/(2j)$.  The resulting accumulated
power is recorded explicitly.

\subsection{Integer moments}

\begin{lemma}\label{lem:integer-moments}
Let $j\geq2$ and $j-1\leq d-2$.  Then
\begin{equation}\label{eq:integer-moment-rank-sum}
 \sum_{x\in\Z^d}r_E(x)^j
 \lesssim_{d,j,\eps}
 R^\eps\sum_{r=0}^{j-1}
 |E|^{r+1}L_r^{j-r-1}R^{d-r-2}.
\end{equation}
Consequently,
\begin{equation}\label{eq:integer-moment}
 \sum_{x\in\Z^d}r_E(x)^j
 \lesssim_{d,j,\eps}
 R^{d-2+\eps}|E|^{j-1}
 +R^{d-j-1+\eps}|E|^j.
\end{equation}
\end{lemma}

\begin{proof}
The claim is trivial for $E=\varnothing$, so assume $|E|\geq1$.
Expanding the $j$-th moment counts ordered configurations
\[
 (x,e_1,\ldots,e_j)\in\Z^d\times E^j,
 \qquad |x-e_i|=R\quad\text{for }1\leq i\leq j.
\]
We organize these configurations according to their affine rank
\[
 r=\dim\Aff\{e_1,\ldots,e_j\},
 \qquad 0\leq r\leq j-1.
\]
Choose $r+1$ tuple positions forming an affine basis $B$.  There are
$O_j(1)$ choices of positions and at most $|E|^{r+1}$ choices of their
values.  Since $B$ is an affine basis of the tuple, every remaining tuple
entry lies in $A_B$.  Since the tuple contributes, it has a common center
$x\in X_B\cap\Z^d$.  Lemma~\ref{lem:common-center-geometry} then implies
that every remaining tuple entry lies on $E\cap\Sigma_B$.  Hence each
remaining entry has at most
\[
 O_{d,j,\eps}\bigl(R^{\eps/(2j)}L_r\bigr)
\]
choices by \eqref{eq:Lr-bound}.

The admissible centers lie in $X_B\cap\Z^d$.  Since the tuple contributes,
$X_B\neq\varnothing$.  By Lemma~\ref{lem:common-center-geometry}, $X_B$ is
either a singleton or a nondegenerate embedded $(d-r-1)$-sphere of radius
at most $R$.  Lemma~\ref{lem:embedded-sphere}, with
$\eta=\eps/(2j)$ and including its $O(1)$ convention for a singleton,
gives
\[
 |X_B\cap\Z^d|
 \lesssim_{d,j,\eps}
 R^{d-r-2+\eps/(2j)}.
\]
There are $j-r-1$ remaining tuple entries, so the accumulated auxiliary
power is
\[
 (j-r-1)\frac{\eps}{2j}+\frac{\eps}{2j}
 =\frac{(j-r)\eps}{2j}
 \leq\frac{\eps}{2}.
\]
Consequently, the contribution of affine rank $r$ is at most
\[
 \lesssim_{d,j,\eps}
 R^{\eps/2}
 |E|^{r+1}L_r^{j-r-1}R^{d-r-2}.
\]
Summing over the $j$ possible ranks and using
$R^{\eps/2}\leq R^\eps$ proves
\eqref{eq:integer-moment-rank-sum}.

It remains to derive \eqref{eq:integer-moment}.  At rank $r=j-1$, the
corresponding term in \eqref{eq:integer-moment-rank-sum} is
\[
 R^{d-j-1+\eps}|E|^j,
\]
which is the second term in \eqref{eq:integer-moment}.

For $r\leq j-2$, consider the ratio
\[
 \frac{|E|^{r+1}L_r^{j-r-1}R^{d-r-2}}
 {R^{d-2}|E|^{j-1}}
 =
 \left(\frac{L_r}{|E|}\right)^{j-r-2}
 \frac{L_r}{R^r}.
\]
Clearly, $L_r\leq|E|$ for every $r$.  We also have $L_r\leq R^r$:
this is immediate for $r=0,1$, while for $r\geq2$ it follows from
$L_r\leq R^{r-2}$ and $R\geq1$.  Thus every rank $r\leq j-2$ is bounded
by the first term in \eqref{eq:integer-moment}.  This proves
\eqref{eq:integer-moment}.
\end{proof}

\section{The even-dimensional low-density estimate}

\begin{theorem}\label{thm:geometric-density}
Let $d\geq4$ be even, let $C_0\geq1$ be fixed, and let
$E\subset\Z^d$ be finite.  If $|E|\leq C_0R^{d/2}$, then
\begin{equation}\label{eq:sparse-moment}
 \sum_{x\in\Z^d}r_E(x)^{\frac{d+2}{2}}
 \lesssim_{d,C_0,\eps}R^{d-2+\eps}|E|^{d/2}.
\end{equation}
Consequently, for every finite $F\subset\Z^d$,
\begin{equation}\label{eq:sparse-incidence}
 I_R(E,F)
 \lesssim_{d,C_0,\eps}
 R^{\frac{2(d-2)}{d+2}+\eps}(|E||F|)^{\frac d{d+2}}.
\end{equation}
\end{theorem}

\begin{proof}[Proof of Theorem~\ref{thm:geometric-density}]
Write $d=2k$.  Then $(d+2)/2=k+1$.  Lemma~\ref{lem:integer-moments} gives
\[
 \sum_xr_E(x)^{k+1}
 \lesssim R^{2k-2+\eps}|E|^k+R^{k-2+\eps}|E|^{k+1}
 \lesssim_{d,C_0,\eps}R^{2k-2+\eps}|E|^k,
\]
where we used $|E|\leq C_0R^k$ in the last inequality.  This proves
\eqref{eq:sparse-moment}.

Finally, H\"older's inequality with conjugate exponents $(d+2)/2$ and
$(d+2)/d$ gives \eqref{eq:sparse-incidence}.
\end{proof}

\section{The four-dimensional bridge estimate}

\begin{theorem}\label{thm:d4-bridge}
Under the localized setup above, assume $d=4$.  This theorem requires only
$R^2\in\mathbb N$.  Let $C_1\geq1$ be fixed, and suppose, after
interchanging $E$ and $F$ if necessary, that
\begin{equation}\label{eq:d4-bridge-range}
 R^2<|E|\leq|F|,
 \qquad
 |E||F|\leq C_1R^5.
\end{equation}
Then
\begin{equation}\label{eq:d4-bridge-incidence}
 I_R(E,F)
 \lesssim_{C_1,\eps}
 R^{2/3+\eps}(|E||F|)^{2/3}.
\end{equation}
\end{theorem}

\begin{proof}
Let $Q$ be the half-open lattice cube from the localization.  Since
$E\cup F\subset Q^*$, also after the possible interchange of $E$ and $F$,
after a simultaneous translation by a lattice vector we may assume that
\[
 \|x\|_\infty\lesssim R
 \qquad\text{for every }x\in E\cup F.
\]
This translation preserves cardinalities and $I_R(E,F)$.

For each $f\in F$, consider the sphere in $\mathbb R^4$ centered at $f$
and of radius $R$.  Its lattice points are precisely $f+S_R$, and
distinct points of $F$ give distinct spheres, so the family has cardinality
$|F|$.  Moreover,
\[
 e\in f+S_R
 \quad\Longleftrightarrow\quad
 |e-f|=R.
\]
Thus the incidences between the points of $E$ and these spheres are counted
exactly by $I_R(E,F)$.  Mudgal's four-dimensional incidence estimate,
in \cite[Section~6, before (6.3)]{Mudgal}, therefore applies.  Its coordinate hypothesis concerns
only the point set and is satisfied here, since
\[
 \|e\|_\infty\lesssim R
 \qquad(e\in E),
\]
and hence $\|e\|_\infty\leq R^2$ for all sufficiently large $R$.  
Applying this estimate gives
\begin{equation}\label{eq:mudgal-d4-local}
 I_R(E,F)
 \lesssim_{\eps}
 R^\eps\left(
 |E|^{9/11}|F|^{8/11}+|E|+|F|
 \right).
\end{equation}

It remains to compare the three terms with the desired scale.  Since
$|E|\leq|F|$,
\[
 |E|^{9/11}|F|^{8/11}
 \leq (|E||F|)^{17/22}.
\]
Using \eqref{eq:d4-bridge-range},
\begin{align*}
 (|E||F|)^{17/22}
 &=(|E||F|)^{2/3}(|E||F|)^{7/66}\\
 &\leq C_1^{7/66}R^{35/66}(|E||F|)^{2/3}\\
 &\leq C_1^{7/66}R^{2/3}(|E||F|)^{2/3}.
\end{align*}
Because the sets are nonempty and $|E|\leq|F|$,
\[
 |E|\leq(|E||F|)^{2/3}
 \leq R^{2/3}(|E||F|)^{2/3}.
\]
Finally,
\begin{align*}
 \frac{|F|}{R^{2/3}(|E||F|)^{2/3}}
 &=\left(\frac{|F|}{R^2|E|^2}\right)^{1/3}\\
 &\leq C_1^{1/3}\frac{R}{|E|}\\
 &\leq C_1^{1/3}R^{-1}.
\end{align*}
Plugging these estimates into \eqref{eq:mudgal-d4-local} proves
\eqref{eq:d4-bridge-incidence}.
\end{proof}

\section{The odd-dimensional low-density estimate}

\begin{theorem}
\label{thm:odd-geometric-density}
Assume $d=2k-1\geq5$, and let $C_1\geq1$ be fixed.  Let
$E,F\subset\Z^d$ be finite and suppose, after interchanging $E$ and $F$ if
necessary, that
\begin{equation}\label{eq:odd-geometric-range}
 |E|\leq |F|,
 \qquad
 |E||F|\leq C_1R^{2d-\kappa_d}.
\end{equation}
Then
\begin{equation}\label{eq:odd-geometric-incidence}
 I_R(E,F)
 \lesssim_{d,C_1,\eps}
 R^{\frac{2(d-2)}{d+2}+\eps}
 (|E||F|)^{\frac d{d+2}}.
\end{equation}
\end{theorem}

\begin{proof}
The claim is trivial if $E=\varnothing$, so assume $|E|\geq1$.  We first
handle the case in which $F$ is much larger than $E$.
If
\begin{equation}\label{eq:odd-first-moment-case}
 |F|^d\geq R^{d(d-2)}|E|^2,
\end{equation}
we just use the trivial first
moment estimate:
\[
 I_R(E,F)
 \leq\sum_{x\in\Z^d}r_E(x)
 \lesssim_d R^{d-2}|E|.
\]
Compared to the desired quantity, the ratio is
\begin{align*}
 \frac{R^{d-2}|E|}
 {R^{\frac{2(d-2)}{d+2}}(|E||F|)^{\frac d{d+2}}}
 &=
 \left(
  \frac{R^{d(d-2)}|E|^2}{|F|^d}
 \right)^{1/(d+2)}
 \leq1.
\end{align*}
Thus \eqref{eq:odd-geometric-incidence} follows in this case.

For the rest of the proof, we may therefore assume
\begin{equation}\label{eq:odd-balanced-case}
 |F|^d\leq R^{d(d-2)}|E|^2.
\end{equation}
Since $k\leq d-2$, H\"older's inequality and
\eqref{eq:integer-moment-rank-sum}, applied with $j=k+1$, give
\begin{align*}
 I_R(E,F)^{k+1}
 &=
 \left(\sum_{x\in F}r_E(x)\right)^{k+1}\\
 &\leq
 |F|^k\sum_{x\in\Z^d}r_E(x)^{k+1}\\
 &\lesssim_{d,\eps}
 R^\eps\sum_{r=0}^{k}
 |F|^k|E|^{r+1}L_r^{k-r}R^{d-r-2}.
\end{align*}
It is enough to show that every rank-$r$ summand, without the auxiliary
$R^\eps$ factor, is bounded by
\[
 \left[
  R^{\frac{2(d-2)}{d+2}}
  (|E||F|)^{\frac d{d+2}}
 \right]^{k+1}.
\]

Since $d=2k-1$, one has
\[
 \frac d{d+2}(k+1)
 =k-\frac1{d+2},
 \qquad
 \frac{2(d-2)}{d+2}(k+1)
 =\frac{(d-2)(d+3)}{d+2}.
\]
Consequently, after division by the preceding target, the normalized
rank-$r$ contribution is
\begin{equation}\label{eq:odd-rank-normalized-exact}
 R^{-\frac{(d-2)(d+3)}{d+2}+d-r-2}
 |E|^{r-k+1+\frac1{d+2}}
 |F|^{\frac1{d+2}}
 L_r^{k-r}.
\end{equation}
The assumption \eqref{eq:odd-balanced-case} gives
\[
 |F|^{1/(d+2)}
 \leq
 R^{\frac{d-2}{d+2}}
 |E|^{\frac{2}{d(d+2)}}.
\]
Using
\[
 \frac1{d+2}+\frac{2}{d(d+2)}=\frac1d,
\]
we see that \eqref{eq:odd-rank-normalized-exact} is at most
\begin{equation}\label{eq:odd-rank-normalized}
 |E|^{r-k+1+\frac1d}L_r^{k-r}R^{-r}.
\end{equation}

For $r=0,1$, since $L_r=1$, the expression in
\eqref{eq:odd-rank-normalized} is
\[
 |E|^{r-k+1+\frac1d}R^{-r}.
\]
Moreover,
\[
 r-k+1+\frac1d
 \leq 2-k+\frac1d
 \leq -1+\frac15<0.
\]
Since $|E|\geq1$ and $R\geq1$, the lowest two rank contributions are
bounded by one.

Now let $2\leq r\leq k-2$.  It remains to bound
\[
 |E|^{r-k+1+\frac1d}L_r^{k-r}R^{-r}.
\]
If $|E|\leq R^{r-2}$, then $L_r=|E|$, and this expression is at most
\[
 |E|^{1+\frac1d}R^{-r}
 \leq
 R^{(r-2)(1+\frac1d)-r}
 =
 R^{-2+\frac{r-2}{d}}.
\]
If $|E|\geq R^{r-2}$, then $L_r=R^{r-2}$.  Since
\[
 r-k+1+\frac1d
 \leq -1+\frac1d<0,
\]
we have
\[
 |E|^{r-k+1+\frac1d}
 \leq
 R^{(r-2)(r-k+1+\frac1d)}.
\]
It follows again that
\begin{align*}
 |E|^{r-k+1+\frac1d}L_r^{k-r}R^{-r}
 &\leq
 R^{(r-2)(r-k+1+\frac1d)}
 R^{(r-2)(k-r)}R^{-r}\\
 &=
 R^{-2+\frac{r-2}{d}}.
\end{align*}
In both cases,
\[
 -2+\frac{r-2}{d}
 \leq
 -2+\frac{k-4}{d}
 =
 -\frac{3d+7}{2d}<0.
\]
Thus every middle-rank contribution is bounded by one.

At rank $r=k-1$, the bound $L_{k-1}\leq R^{k-3}$ shows that
\eqref{eq:odd-rank-normalized} is at most
\[
 |E|^{1/d}L_{k-1}R^{1-k}
 \leq |E|^{1/d}R^{-2}.
\]
Since $|E|\leq|F|$ and \eqref{eq:odd-geometric-range} holds,
\[
 |E|^2\leq |E||F|
 \leq C_1R^{2d-\kappa_d}.
\]
Therefore
\[
 |E|^{1/d}R^{-2}
 \leq
 C_1^{1/(2d)}
 R^{\frac{2d-\kappa_d}{2d}-2}
 =
 C_1^{1/(2d)}R^{-1-\frac{\kappa_d}{2d}}
 \lesssim_{d,C_1}1.
\]

At the top rank $r=k$, we return to
\eqref{eq:odd-rank-normalized-exact}.  Since
\[
 d-k-2=\frac{d-5}{2},
\]
the top-rank contribution is
\begin{align*}
 &R^{-\frac{(d-2)(d+3)}{d+2}+\frac{d-5}{2}}
 |E|^{1+\frac1{d+2}}|F|^{\frac1{d+2}}\\
 &\qquad=
 R^{-\frac{(d-2)(d+3)}{d+2}+\frac{d-5}{2}}
 |E|(|E||F|)^{\frac1{d+2}}.
\end{align*}
Since $|E|\leq|F|$,
\[
 |E|\leq(|E||F|)^{1/2},
\]
and hence the last display is at most
\[
 R^{-\frac{(d-2)(d+3)}{d+2}+\frac{d-5}{2}}
 (|E||F|)^{\frac{d+4}{2(d+2)}}.
\]
Using \eqref{eq:odd-geometric-range}, this is bounded by
\[
 C_1^{\frac{d+4}{2(d+2)}}
 R^{-\frac{(d-2)(d+3)}{d+2}
   +\frac{d-5}{2}
   +(2d-\kappa_d)\frac{d+4}{2(d+2)}}.
\]
By the definition of $\kappa_d$,
\[
 -\frac{(d-2)(d+3)}{d+2}
 +\frac{d-5}{2}
 +(2d-\kappa_d)\frac{d+4}{2(d+2)}
 =
 -\frac{d-5}{d-2}\leq0.
\]
Thus the top-rank contribution is also $O_{d,C_1}(1)$.

We have proved that every rank contribution is bounded by the desired
quantity.  Summing the $k+1=O_d(1)$ ranks gives
\[
 I_R(E,F)^{k+1}
 \lesssim_{d,C_1,\eps}
 R^{\eps}
 \left[
  R^{\frac{2(d-2)}{d+2}}
  (|E||F|)^{\frac d{d+2}}
 \right]^{k+1}.
\]
This proves
\eqref{eq:odd-geometric-incidence}.
\end{proof}

\section{Proof of Proposition~\ref{prop:local-endpoint}}
\label{sec:local-synthesis}

The four theorems: Theorem~\ref{thm:analytic-density},
Theorem~\ref{thm:geometric-density},
Theorem~\ref{thm:d4-bridge},
Theorem~\ref{thm:odd-geometric-density} now cover every local density regime.

\begin{proof}[Proof of Proposition~\ref{prop:local-endpoint}]
Recall the exact definition
\[
 \delta=\frac{|E||F|}{R^{2d}}.
\]
If $\delta\geq R^{-\kappa_d}$,
Theorem~\ref{thm:analytic-density} gives
\eqref{eq:local-target-intro}.  Suppose instead that
$\delta<R^{-\kappa_d}$.  Then
\begin{equation}\label{eq:sparse-product-range}
 |E||F|<R^{2d-\kappa_d}.
\end{equation}

Suppose first that $d=4$.  Directly from
\eqref{eq:kappa-def}, $\kappa_4=3$, and hence
\[
 |E||F|<R^5.
\]
Interchange $E$ and $F$ if necessary so that $|E|\leq|F|$.
If $|E|\leq R^2$, Theorem~\ref{thm:geometric-density} applies with
$C_0=1$.  If $|E|>R^2$, Theorem~\ref{thm:d4-bridge} applies with
$C_1=1$.

Now suppose that $d$ is even and $d\geq6$.  Since
\begin{equation}\label{eq:kappa-minus-d}
 \kappa_d-d=\frac{d-6}{d-2}\geq0,
\end{equation}
\eqref{eq:sparse-product-range} gives
\[
 |E||F|<R^{2d-\kappa_d}\leq R^d,
 \qquad
 \min\{|E|,|F|\}<R^{d/2}.
\]
By symmetry, Theorem~\ref{thm:geometric-density} applies with the smaller
set in the role of $E$ and with $C_0=1$.

Finally, suppose that $d$ is odd.  Interchange $E$ and $F$ if necessary so
that $|E|\leq|F|$.  Theorem~\ref{thm:odd-geometric-density} applies to
\eqref{eq:sparse-product-range} with $C_1=1$.

In every geometric case the incidence estimate is exactly equivalent to
\eqref{eq:local-target-intro}.  Indeed,
\eqref{eq:calA-rE} gives
\[
 R^{-d}\langle\cA_R\one_E,\one_F\rangle
 =R^{2-2d}I_R(E,F),
\]
and substituting
$|E||F|=R^{2d}\delta$ into
\[
 I_R(E,F)\lesssim_{d,\eps}
 R^{\frac{2(d-2)}{d+2}+\eps}
 (|E||F|)^{\frac d{d+2}}
\]
produces precisely $R^\eps\delta^{d/(d+2)}$.  This completes the proof.
\end{proof}

\noindent\textbf{Completion of the global restricted estimate.}
Proposition~\ref{prop:local-endpoint} is precisely the local estimate required
in the localization reduction preceding \eqref{eq:local-target-intro}.
Therefore the global restricted estimate \eqref{eq:global-restricted} holds.
Equivalently,
\begin{equation}\label{eq:restricted-lorentz}
 \|\cA_R\|_{\ell^{\frac{d+2}{d},1}\to\ell^{\frac{d+2}{2},\infty}}
 \lesssim_{d,\eps}R^{-\frac{d(d-2)}{d+2}+\eps}.
\end{equation}

For $d\geq5$, \eqref{eq:sphere-cardinality} transfers this estimate to
$A_R$.  In dimension four, the same conclusion follows from
\eqref{eq:four-square-lower}.  Hence
\begin{equation}\label{eq:restricted-lorentz-probability}
 \|A_R\|_{\ell^{\frac{d+2}{d},1}\to\ell^{\frac{d+2}{2},\infty}}
 \lesssim_{d,\eps}R^{-\frac{d(d-2)}{d+2}+\eps}.
\end{equation}

\section{Interpolation and completion}

To complete the proof, we combine
\eqref{eq:restricted-lorentz-probability} with the elementary bounds
\[
 \|A_R\|_{\ell^1\to\ell^\infty}
 =|S_R|^{-1}
 \lesssim_d R^{-(d-2)},
 \qquad
 \|A_R\|_{\ell^2\to\ell^2}\leq1.
\]
The first bound follows from \eqref{eq:sphere-cardinality} when $d\geq5$
and from \eqref{eq:four-square-lower} when $d=4$; the second follows because
$A_R$ is convolution with a probability measure.

A standard two-step interpolation argument upgrades
\eqref{eq:restricted-lorentz-probability} to strong type at the endpoint.
First, off-diagonal Marcinkiewicz interpolation between the restricted
endpoint, invoked with a sufficiently small fraction of the prescribed
$\eps$ loss, and the $\ell^1\to\ell^\infty$ bound gives a strong estimate
at a nearby point on the conjugate line.  Interpolating this estimate with
the $\ell^2$ bound returns to the endpoint.  Since
$d(2/p-1)$ is affine in $1/p$ and vanishes at $p=2$, the principal powers
of $R$ interpolate exactly, while the additional loss is absorbed into
$R^\eps$.  Hence
\begin{equation}\label{eq:strong-endpoint}
 \|A_R f\|_{\ell^{\frac{d+2}{2}}}
 \lesssim_{d,\eps}
 R^{-\frac{d(d-2)}{d+2}+\eps}
 \|f\|_{\ell^{\frac{d+2}{d}}}.
\end{equation}

Interpolating \eqref{eq:strong-endpoint} with the $\ell^2$ bound gives
\[
 \|A_R f\|_{\ell^{p'}}
 \lesssim_{d,p,\eps}
 R^{-d(2/p-1)+\eps}\|f\|_{\ell^p},
 \qquad
 \frac{d+2}{d}\leq p\leq2.
\]
This is \eqref{eq:main-estimate}.  Remark~\ref{rem:sharpness} proves the
asserted sharpness and completes the proof of Theorem~\ref{thm:main}.

\begin{proof}[Proof of Corollary~\ref{cor:fixed-distance-incidence}]
By \eqref{eq:calA-rE} and the global restricted estimate
\eqref{eq:global-restricted},
\begin{align*}
 I_R(E,F)
 &=R^{d-2}\langle\cA_R\one_E,\one_F\rangle\\
 &\lesssim_{d,\eps}
 R^{d-2-\frac{d(d-2)}{d+2}+\eps}
 (|E||F|)^{\frac d{d+2}}\\
 &=R^{\frac{2(d-2)}{d+2}+\eps}
 (|E||F|)^{\frac d{d+2}}.
\end{align*}
This is \eqref{eq:fixed-distance-incidence}.  Taking $F=E$ gives the final
assertion.
\end{proof}

\section*{Acknowledgments}
R.H. was partially supported by NSF DMS-2143369.
ChatGPT was used to assist the authors with computing the integer moment
in Lemma~\ref{lem:integer-moments},
and minor editing and polishing throughout the paper.

\end{document}